\documentclass[10pt]{amsart}

\usepackage{amssymb,latexsym,amsmath,amsthm}
\usepackage{enumerate}
\usepackage[latin1]{inputenc}

\newtheorem{thm}{Theorem}[section]
\newtheorem{lem}[thm]{Lemma}
\newtheorem{prop}[thm]{Proposition}

\theoremstyle{definition}
\newtheorem{defin}[thm]{Definition}

\numberwithin{equation}{section}

\newcommand{\floor}[1]{\left[ #1 \right]}
\newcommand{\nrm}[1]{\mathord{\left\lVert #1 \right\rVert}}

\newcommand{\N}{\mathbb{N}}
\newcommand{\Z}{\mathbb{Z}}
\newcommand{\CN}{\mathcal{N}}
\newcommand{\CA}{\mathcal{A}}
\newcommand{\eand}{\,\,\,\text{ and }\,\,\,}
\newcommand{\eps}{\varepsilon}
\newcommand{\lset}{\left\lbrace}
\newcommand{\rset}{\right\rbrace}
\newcommand{\abs}[1]{\mathord{\left|#1\right|}}

\newcommand{\mcA}{\mathcal{A}}

\begin{document}

\title[Weighted Brun-Titchmarsh Inequality]{A Brun-Titchmarsh inequality for weighted sums over prime numbers}

\author[J. B\"uthe]{Jan B\"uthe}
\address{Universit\"at Bonn, Endenicher Allee 60
53115 Bonn, Germany}
\email{jbuethe@math.uni-bonn.de}
\date{\today}

\begin{abstract}
We prove explicit upper bounds for weighted sums over prime numbers in arithmetic progressions with slowly varying weight functions. The results generalize the well-known Brun-Titchmarsh inequality.
\end{abstract}

\subjclass[2010]{Primary 11N36; Secondary 11N13}

\keywords{sums over prime numbers, Brun-Titchmarsh inequality}

\maketitle

\section{Introduction}
In this paper we prove upper bounds for sums of the form
\begin{equation}\label{e:weighted-sum}
	\sum_{\substack{p\in I\\p\equiv l\bmod k}} f(p),
\end{equation}
where $I=[x,x+y]\subset[0,\infty)$ is an interval, $k$ and $l$ are coprime integers, and $f$ is a slowly varying weight function on $I$. The results generalize the well-known Brun-Titchmarsh inequality for the number of prime numbers in arithmetic progressions \cite{Titchmarsh1930,Iwaniec1982}.

The work is motivated by the following problem. Functions as the Riemann prime counting function $\pi^*(x)$ or the Chebyshov function $\psi(x)$ satisfy certain explicit formulas involving sums over the zeros of the Riemann zeta function \cite{Riemann59,Mangoldt85}. If one is interested in studying these functions via their explicit formulas, one has to deal with the problem that these sums converge only due to oscillation and are therefore difficult to handle. A natural way to overcome this problem is to study continuous (or smooth) approximations to the functions $\pi^*(x)$ and $\psi(x)$ \cite{Faber2013,Buethe2014}. The problem of estimating the approximation error then leads to sums of the form \eqref{e:weighted-sum}. As far as the author knows, this has been carried out in \cite{RS03} for the first time, to obtain short effective intervals containing prime numbers, where the Brun-Titchmarsh inequality is used to estimate such sums.

Apart from that, the results are likely to have other applications. One further application might be to improve the estimates for the error term in the asymptotic formula for the zero counting function of the Riemann zeta function, where similar sums occur \cite{Carneiro2012}.

 The proof of the results is elementary and uses a simple notion of weighted sieves.

\section{Notations}
We use the notations in \cite{Montgomery2006} in the context of number theory and those in \cite{Adams2003} for function spaces and their norms. In \eqref{e:weighted-sum} we will consider functions  $f\in W^{1,1}(I)$, the closure of $C^1(I)$ with respect to the norm
\[
\nrm{f}_{1,1,I} = \nrm{f}_{1,I} + \nrm{f'}_{1,I} = \int_I \abs{f(t)}\, dt + \int_I \abs{f'(t)}\, dt,
\]
which we regard as a subspace of $C^0(I)$ by the continuous embedding of $C^{1}(I)$, equipped with the norm $\nrm{\cdot}_{1,1,I}$, into $C^0(I)$.

Furthermore, we use the notations $I=[x,x+y]\subset [0,\infty)$, $\mcA = I\cap (l + k\Z)$, $\mcA_d = \{a\in \mcA \mid d|a\}$ and 
\[
 P(z,k) = \prod_{\substack{p\leq z\\p\nmid k}} p
 \]
throughout this paper. For explicit estimates, we also use Turing's $\Theta$-notation and say $f(x) = \Theta(g(x))$ for $x\in A$ iff $\abs{f(x)} \leq g(x)$ for all $x\in A$.

\section{Weighted Sieves}
The weighted sieves are based on the following lemma.
\begin{lem}\label{l:rs}
 Let $f\in W^{1,1}(I)$ and let $d,k\in\N$ satisfy $(d,k)=1$. Then we have
\begin{equation*}
 \Biggl| \frac{1}{kd}\int_I f(t)\, dt - \sum_{\substack{n\in I\cap(l+k\Z)\\d\mid n}} f(n)\Biggr| \leq \nrm{f}_{\infty,I} + \nrm{f'}_{1,I}.
\end{equation*}
\end{lem}

%%%%%%%%%%%%%%%%%%%%%%%%%%%%%%%%%%%%%%%%%%%%%%%%%%%%%%%%%%%%%%%%%%%%
\begin{proof} The proof is essentially Gallagher's proof of the large sieve inequality, as noted by the referee. Let
\begin{equation}\label{e:rd-def}
 r_d =  \sum_{n\in \CA_d} f(n)-\frac{1}{kd}\int_I f(t)\, dt.
\end{equation}
We have to show $\abs{r_d} \leq \nrm f_{\infty,I} + \nrm {f'}_{1,I}.$ By replacing $k$ by $kd$ and adjusting $l$ we may assume $d=1$, and by replacing $f$ by $f(k\cdot+l)$ and adjusting $x$ and $y$ we may also assume $k=1$.

Now let $x' = x+\floor{ y}$, $I'=[x,x')$ and $I'' = [x',x+y]$. We define
\begin{align*}
 r' &= \sum_{a\in \CA\cap I'} f(a) - \int_{I'} f(t)\, dt, \\
 r'' &= \sum_{a\in \CA\cap I''} f(a) - \int_{I''} f(t)\, dt.
\end{align*}
Then we have $r_1 = r' + r''$, since $I$ is the disjoint union of $I'$ and $I''$.

Now let $a<b$ and let $\xi\in [a,b]$. Then, by choosing an approximating sequence of $C^1(I)$-functions, we can extend the well-known identity
\begin{equation*}\label{e:int-eq}
 \int_a^b g(t)\, dt = (b-a)g(\xi)  + \int_a^\xi g'(t)(a-t)\,dt + \int_\xi^b g'(t) (b-t)\,dt
\end{equation*}
to all $g\in W^{1,1}((a,b))$. We therefore obtain the bound
\begin{equation}\label{e:int-bound}
 \abs{g(\xi) - \frac{1}{b-a}\int_{a}^b g(t)\,dt} \leq \int_a^b \abs{g'(t)}\, dt.
\end{equation}

The interval $I'$  is the disjoint union of half-open intervals $[x+n,x+n+1)$, $n=0,1,\dots, \floor{y}-1$.
 Since each such interval contains exactly one element of $\CA$, we can apply \eqref{e:int-bound} which yields the bound
\begin{equation*}\label{e:r'd}
 \abs{r'} \leq \int_{I'} \abs{f'(t)}\,dt.
\end{equation*}

It remains to estimate $r''$. Since we have $\abs{I''}<1$, the intersection $\CA\cap I''$ is either empty or contains exactly one element $a$. In the first case we have
\begin{equation*}
 \abs{r''} \leq \int_{I''} f(t)\,dt \leq \nrm{f}_{\infty, I''} 
\end{equation*}
and in the second case we have
\begin{align*}
 \abs{r''} &\leq \abs{ \int_{I''} f(t)\, dt -\abs{I''}f(a)} + \left(1-\abs{I''}\right) \abs{f(a)} \\
&\leq  \abs{I''} \int_{I''} \abs{f'(t)}\, dt + \left(1-\abs{I''}\right) \abs{f(a)}\leq \nrm{f'}_{1,I''} + \nrm{f}_{\infty, I''}.
\end{align*}
So in both cases $\abs{r''} \leq \nrm{f'}_{1,I''} + \nrm{f}_{\infty, I''}$ holds. We thus get
\[
\abs{r_1} \leq  \abs{r'} + \abs{r''} \leq \nrm{f'}_{1,I'} + \nrm{f'}_{1,I''} + \nrm{f}_{\infty, I''} \leq  \nrm f_{\infty,I} + \nrm {f'}_{1,I}.
\]
\end{proof}

%%%%%%%%%%%%%%%%%%%%%%%%%%%%%%%%%%%%%%%%%%%%%%%%%%%%%%

\begin{prop}[Weighted Selberg Sieve]\label{s:wss}
Let $f\in W^{1,1}(I)$ be non-negative and let $z\geq 1$. We define
\begin{equation*}
 S_k(z) = \sum_{\substack{n\leq z\\(n,k)=1}} \frac{\mu(n)^2}{\varphi(n)}\eand H_k(z) = \sum_{\substack{n\leq z\\ (n,k)=1}} \mu(n)^2 \frac{\sigma(n)}{\varphi(n)}.
\end{equation*}
Then the inequality
\begin{equation*}
 \sum_{\substack{n\in I\cap(l+k\Z)\\(n,P(z,k))=1}} f(n) \leq \frac{\nrm{f}_{1,I}}{kS_k(z)} + (\nrm{f}_{\infty,I} + \nrm{f'}_{1,I}) \frac{H_k(z)^2}{S_k(z)^2}
\end{equation*}
holds.
\end{prop}

\begin{proof}
The proof is based on Selberg's lambda squared method \cite{Selberg1947}. Let $r_d$ be as in \eqref{e:rd-def} and let 
\[\CN_k = \lset n\in \N\mid 1\leq n\leq z, \mu(n)\neq 0, (n,k) = 1\rset.\]
Let $(\lambda_n)_{n\in\CN_k}$ be a sequence of real numbers satisfying $\lambda_1 = 1$. Then we have
\begin{equation}\label{e:selberg}
\sum_{a\in\CA} f(a) \Bigl(\sum_{\substack{n\in\CN_k\\ n\mid a}} \lambda_n\Bigr)^2 
\geq \sum_{\substack{a\in \CA\\(a,P(z,k))=1}} f(a)\Bigl(\sum_{\substack{n\in\CN_k\\ n\mid a}} \lambda_n\Bigr)^2 =  \sum_{\substack{a\in \CA\\(a,P(z,k))=1}} f(a).
\end{equation}
By a simple calculation we get
\begin{align}
 \sum_{a\in\CA} f(a) \Bigl(\sum_{\substack{n\in\CN_k \\ n\mid a}} \lambda_n\Bigr)^2 
&= \sum_{(n_1,n_2)\in\CN_k^2} \lambda_{n_1}\lambda_{n_2} \sum_{a\in\CA_{[n_1,n_2]}} f(a)\notag\\
&= \frac{\nrm{f}_{1,I}}{k} \,\,\sum_{(n_1,n_2)\in\CN_k^2} \frac{\lambda_{n_1}\lambda_{n_2}}{[n_1,n_2]}
 + \!\! \sum_{(n_1,n_2)\in\CN_k^2} \lambda_{n_1}\lambda_{n_2} r_{[n_1,n_2]}.\label{e:qform+r}
\end{align}
Now \eqref{e:selberg}, \eqref{e:qform+r} and Lemma \ref{l:rs} together imply
\begin{equation*}
  \sum_{\substack{n\in I\cap(l+k\Z)\\(n,P(z,k))=1}} f(n) \leq \frac{\nrm{f}_{1,I}}{k} \sum_{(n_1,n_2)\in\CN_k^2} \frac{\lambda_{n_1}\lambda_{n_2}}{[n_1,n_2]} + (\nrm f_{\infty,I} + \nrm f' _1) \Bigl(\sum_{n\in\CN_k} \abs{\lambda_n}\Bigr)^2.
\end{equation*}
Here we choose the well-known minimizing sequence
\begin{equation*}
 \lambda_n = \frac{\mu(n)n}{S_k(z)} \sum_{\substack{m\in\CN_k\\n|m}} \frac{1}{\varphi(n)}
\end{equation*}
for the quadratic form in \eqref{e:qform+r}, which gives
\begin{equation*}
\sum_{(n_1,n_2)\in\CN_k^2}\frac{\lambda_{n_1}\lambda_{n_2}}{[n_1,n_2]} =  \frac{1}{S_k(z)}
\eand
\sum_{n\in\CN_k}\abs{\lambda_n} = \frac{H_k(z)}{S_k(z)}
\end{equation*}
(see \cite{Lint1965}).
\end{proof}

%%%%%%%%%%%%%%%%%%%%%%%%%%%%%%%%%%%%%%%%%%%%%%%%%%%%%%%%%%%%%%%%%

\begin{prop}[Weighted Eratosthenes Sieve]\label{s:wes}
 Let $f\in W^{1,1}(I)$ be non-negative and let $z\geq 1$. Then we have
\begin{equation*}
  \sum_{\substack{n\in I\cap(l+k\Z)\\(n,P(z,k))=1}} f(n) \leq \frac{\nrm{f}_{1,I}}{k}\prod_{\substack{p\leq z\\p\nmid k}} \left(1-\frac{1}{p}\right) + (\nrm{f}_{\infty,I} + \nrm{f'}_{1,I})\, 2^{\pi(z)}.
\end{equation*}
\end{prop}

\begin{proof}
 We have
\begin{align*}
\sum_{\substack{n\in I\cap(l+k\Z)\\(n,P(z,k))=1}} f(n) 
  &= \sum_{a\in\CA} f(a) \sum_{d\mid (a,P(z,k))} \mu(d)  \notag \\
  &= \sum_{d\mid P(z,k)} \mu(d) \sum_{a\in \CA_d} f(a) \notag \\
  &= \frac{\nrm f_{1,I}}{k}\sum_{d\mid P(z,k)}\frac{\mu(d)}{d} + \sum_{ d\mid P(z,k)} \mu(d) r_d \notag \\
  &\leq \frac{\nrm f_{1,I}}{k} \prod_{p | P(z,k)} \left(1-\frac{1}{p}\right) + (\nrm f_{\infty,I} + \nrm {f'}_{1,I})\, 2^{\pi(z)}.\qedhere
\end{align*}
\end{proof}

\section{The weighted Brun-Titchmarsh Inequality}

We first prove a general version of the weighted Brun-Titchmarsh inequality based on the considerations in \cite{Lint1965}. We then give a stronger result for the case $k=1$, which is of special interest for the prime counting function. The first result implies the Brun-Titchmarsh inequality as stated in \cite{Lint1965}, and the second result implies the stronger version in \cite{MV1973}.

\begin{defin}
We define the functional $\rho_I\colon W^{1,1}(I)\rightarrow [0,\infty)$ by
\begin{equation*}
 \rho_I(f) =      \frac{\nrm{f}_{1,I}}{\nrm{f}_{\infty,I} + \nrm{f'}_{1,I}}
\end{equation*}
for $f\in W^{1,1}(I)\setminus\{0\}$ and $\rho_I(0) = 0$.
\end{defin}

\begin{thm}\label{t:wbti1}
 Let $I=[x,x+y]\subset[0,\infty)$, and let $l,k\in\N$ be coprime numbers. Then the inequalities
\begin{equation*}\label{e:wbti2}
 \sum_{\substack{p\in I\\ p\equiv l \bmod k}} f(p) < 2 \frac{\nrm{f}_{1,I}}{\varphi(k)\log (\rho_I(f)/k)}\left( 1+ \frac{8}{\log (\rho_I(f)/k)} \right)
\end{equation*}
and
\begin{equation*}\label{e:wbti3}
 \sum_{\substack{p\in I\\ p\equiv l \bmod k}} f(p) < 3 \frac{\nrm{f}_{1,I}}{\varphi(k)\log (\rho_I(f)/k)} 
\end{equation*}
hold for all non-negative $f\in W^{1,1}(I)$ satisfying $\rho_I(f)>k$.
\end{thm}

\begin{proof}
We reduce the proof to a situation in the proof of the ordinary Brun-Titchmarsh inequality in \cite{Lint1965}.

Let $g = f/(\nrm{f}_{\infty,I}+\nrm{f'}_{1,I})$. Then we have $\rho_I(g) = \rho_I(f)$ and $\nrm{g}_{1,I} = \rho_I(f)$. We define $Y = \rho_I(f)$. It then suffices to prove the inequalities
\begin{equation}\label{e:wbtu-pb1}
\sum_{\substack{p\in I\\p\equiv l\bmod k}} g(p) < 2 \frac{Y}{\varphi(k)\log(Y/k)}\left( 1 + \frac{8}{\log(Y/k)}\right) 
\end{equation}
and
\begin{equation}\label{e:wbtu-pb2}
\sum_{\substack{p\in I\\p\equiv l\bmod k}} g(p) < 3 \frac{Y}{\varphi(k)\log(Y/k)}.
\end{equation}

Let $k_1=[k,2]$ and let $l_1$ be a representative of the odd pre-image (i.e. the pre-image containing only odd representatives) of the residue class of $l$ under the projection $\Z/k_1\Z \rightarrow \Z/k\Z$. Then, since $\nrm{g}_{\infty,I} \leq 1$, we have
\begin{equation*}
 \sum_{\substack{p\in I\\p\equiv l\bmod k}} g(p) \leq  \sum_{\substack{p\in I \\p\equiv l_1\bmod k_1}} g(p) + 1.
\end{equation*}
Now Proposition \ref{s:wss} gives the bound
\begin{equation}\label{e:t2-start}
\sum_{\substack{p\in I\\p\equiv l\bmod k}} g(p) \leq \frac{Y}{\varphi(k) S_1(z)} + \frac{H_{k_1}^2(z)}{S_{k_1}^2(z)} + \pi(z,k_1,l_1) + 1.
\end{equation}
We are now in the situation, where we have to show that the right hand side of \eqref{e:t2-start} is bounded by either \eqref{e:wbtu-pb1} or \eqref{e:wbtu-pb2} for a suitable choice of $z$. But this is carried out in \cite{Lint1965}.
\end{proof}

From Theorem \ref{t:wbti1} one recovers the ordinary Brun-Titchmarsh inequality in the form 
\[
\pi(x+y;k,l) - \pi(x;k,l) < \frac{2 y}{\varphi(k)\log(y/k)}\left(1+ \frac{8}{\log(y/k)}\right)
\]
by taking $f\equiv 1$ on $[x,x+y]\subset[0,\infty)$, which is slightly weaker than the strongest version proved in \cite{MV1973}.

%%%%%%%%%%%%%%%%%%%%%%%%%%%%%%%%%%%%%%
\begin{thm}\label{t:wbti2}
Let $I=[x,x+y]\subset[0,\infty)$. Then we have
\begin{equation*}\label{e:wbti1}
 \sum_{p\in I} f(p) < 2 \frac{\nrm{f}_{1,I}}{\log (\rho_I(f))} 
\end{equation*}
for all non-negative $f\in W^{1,1}(I)$ satisfying $\rho_I(f)>1$.
\end{thm}

The proof is based on sharper estimates for $H_1$ and $S_1$, provided by the following two lemmas. The second lemma is an explicit version of a result of Ward \cite{Ward1927}.

%%%%%%%%%%%%%%%%%%%%%%%%%%%%%%%%%%
\begin{lem}\label{l:H-asymp}
We have
\begin{equation*}
 H_1(z) = \frac{15}{\pi^2} \,z + \Theta(47 \sqrt{z})
\end{equation*}
for all $z\geq 1$.
\end{lem}

\begin{proof}
We will be needing the following well-known identities for the Riemann zeta function (see e. g. \cite{Titchmarsh1951}):

\begin{gather}
 \zeta(2) = \frac{\pi^2}{6}, \;\;\; \zeta(4) = \frac{\pi^4}{90}, \quad \frac{\zeta(s)}{\zeta(2s)} = \prod_{p} (1+p^{-s}) = \sum_{n} \frac{\mu(n)^2}{n^s},\quad \notag\\
 \sum_{n} \frac{2^{\omega(n)}}{n^s} = \frac{\zeta(s)^2}{\zeta(2s)} \eand \frac{1}{\zeta(s)} = \sum_{n} \frac{\mu(n)}{n^s}.\label{e:zeta-identities}
\end{gather}
We define
\begin{align*}
 h(s) =& \prod_p \left(1 + \frac{2}{(1+p^{s})(p-1)}\right),\\
 \tilde h(s) =& \prod_p \left(1 + \frac{2}{(p^{s}-1)(p-1)}\right).
\end{align*}
In $\Re(s)>0$ these products converge normally and we have
\begin{equation*}
h(s) = \sum_n c_n n^{-s} \eand \tilde h(s) = \sum_n \abs{c_n} n^{-s}
\end{equation*}
for suitable $c_n$. From \eqref{e:zeta-identities} we get
\begin{align*}
 \frac{\zeta(s)}{\zeta(2s)}h(s) 
 &= \prod_p (1+p^{-s})  \prod_p \left(1 + \frac{2}{(1+p^{s})(p-1)}\right) \\
 &= \prod_p\left(1+p^{-s}\frac{p+1}{p-1}\right) = \sum_n \frac{\mu(n)^2\sigma(n)}{\varphi(n)} n^{-s},
\end{align*}
which gives
\begin{equation}\label{e:Rz-id1}
 \sum_{n\leq z} \frac{\mu(n)^2\sigma(n)}{\varphi(n)} = \sum_{d\leq z} c_d \sum_{m\leq z/d} \mu(m)^2.
\end{equation}

Combining \cite[Theorem 3]{CDE07} with a short computar calculation, we obtain the bound
\begin{equation}\label{e:Q-bound}
Q(z) :=  \sum_{n\leq z} \mu(n)^2 = \frac{6}{\pi^2} z + \Theta(0.68 \sqrt{z})
\end{equation}
for $z\geq 1$.

This and \eqref{e:Rz-id1} together imply
\begin{align}
 H_1(z) &= \sum_{d\leq z} c_d\left(\frac{6z}{\pi^2 d} + \Theta\left(0.68 \sqrt{z/d}\right)\right) \notag\\
&= \frac{6z}{\pi^2}\sum_d \frac{c_d}{d} + \Theta\left(0.68 \sqrt{z}\sum_d \frac{\abs{c_d}}{\sqrt{d}} + \frac{6z}{\pi^2}\sum_{d>z}\frac{\abs{c_d}}{d}\right)\notag\\
&= \frac{6}{\pi^2}h(1) z + \Theta(1.3\cdot \tilde h(1/2)\sqrt{z}).\label{e:R-bound1}
\end{align}

The value $h(1)$ is given by
\begin{equation*}
 h(1) = \prod_p \left(1 + \frac{2}{p^2-1}\right) = \sum_n \frac{2^{\omega(n)}}{n^2} = \frac{\zeta(2)^2}{\zeta(4)} = \frac{5}{2}.
\end{equation*}

Next we estimate $\tilde h(1/2)$. For $t\geq 16$ the inequality $(\sqrt{t}-1)(t-1)\geq \frac23 t^{3/2}$ holds and we therefore have
\begin{equation}\label{e:h-half-bound}
 \tilde h(1/2) \leq \left(\prod_{p<10000} \frac{1+2(\sqrt{p}-1)^{-1}(p-1)^{-1}}{(1+p^{-3/2})^3} \right) 
 \frac{\zeta(3/2)^3}{\zeta(3)^3}.
\end{equation}

Here, the product on the right hand side is bounded by $3.5$ and  we have $\zeta(3/2)^3/\zeta(3)^3\leq 10.27$. Therefore, we obtain the bound $\tilde h(1/2) \leq 36$. Inserting this bound in \eqref{e:R-bound1} yields the assertion.

\end{proof}

%%%%%%%%%%%%%%%%%%%%%%%%%%%%%%%%%%%%%%%%%%%%%%%
\begin{lem}\label{l:S-asymp}
 Let $z\geq 10^9$. Then we have
\begin{equation}\label{e:S-asymp}
S_1(z) = \log(z) + C_0 + \sum_p \frac{\log(p)}{p(p-1)} + \Theta\left(\frac{58}{\sqrt{z}}\right).
\end{equation}
\end{lem}
%%%%%%%%%%%%%%%%%%%%%%%%%%%%%%%%%%%%%%%%%%%%%%%%

Here, $C_0= 0.5772156\dots$ denotes the Euler-Mascheroni constant.

\begin{proof}
The proof is similar to the proof of Lemma \ref{l:H-asymp}. Let $Q(z)$ be as in \eqref{e:Q-bound} and let $R(z) = Q(z) - \frac 6{\pi^2} z$. Then partial summation yields
\begin{equation}
 \sum_{n\leq z} \frac{\mu(n)^2}{n} = \frac{6}{\pi^2}\log (z) + \frac{Q(z)}{z} + \int_{1}^\infty \frac{R(t)}{t^2}\, dt - \int_{z}^\infty \frac{R(t)}{t^2}\, dt.
\end{equation}
Consequently, we have
\begin{equation}\label{e:aux1}
 \sum_{n\leq z} \frac{\mu(n)^2}{n} = \frac{\log (z)}{\zeta(2)} + A + \Theta\left(\frac{2.04}{\sqrt{z}}\right),
\end{equation}
where the constant is given by
\[
A = \frac{C_0}{\zeta(2)}-2\sum_d \frac{\mu(d)\log(d)}{d^2} = \frac{C_0}{\zeta(2)}-2\frac{\zeta'(2)}{\zeta(2)^2}
\]
(see e.g. \cite[Lemma 5.4.2]{Bruedern1995}).

Now let
\begin{align*}
 h(s) &= \prod_p\left(1+\frac{1}{(1+p^{s})(p-1)}\right), \\
 \tilde h(s) &= \prod_p\left(1+\frac{1}{(p^{s}-1)(p-1)}\right).
\end{align*}
Then we have
\begin{equation}\label{e:S-gen-ep}
 \frac{\zeta(s)}{\zeta(2s)} h(s) = \sum_n \frac{n\,\mu(n)^2}{\varphi(n)}n^{-s}
\end{equation}
and there are $c_n$ such that
\begin{equation*}
h(s) = \sum_n c_n n^{-s} \eand \tilde h(s) = \sum_n \abs{c_n} n^{-s}
\end{equation*}
holds in $\Re(s)>0$. Combining \eqref{e:S-gen-ep} and \eqref{e:aux1} we now get
\begin{align}\label{e:S-bound1}
 \sum_{n\leq z} \frac{\mu(n)^2}{\varphi(n)} &= \sum_{d\leq z} \frac{c_d}{d} \sum_{m\leq z/d} \frac{\mu(m)^2}{m}  \\
&= \sum_{d\leq z} \left(\frac{\log(z/d)}{\zeta(2)} + A\right) \frac{c_d}{d} + \Theta\left(\tilde h\Bigl(\frac 12\Bigr) \frac{2.04}{\sqrt{z}}\right)\notag.
\end{align}
Since 
\begin{equation*}
 h(1) = \prod_p\left( 1+p^{-2}\frac{1}{1-p^{-2}}\right) = \zeta(2),
\end{equation*}
we have
\begin{equation}\label{e:rm2}
\sum_{d\leq z} \left(\frac{\log(z/d)}{\zeta(2)} + A\right) \frac{c_d}{d} = \log(z) + B + \Theta\left( \sum_{d>z}\frac{\abs{c_d}}{d} \left(\frac{\log(dz)}{\zeta(2)} + \abs{A}\right)  \right)
\end{equation}
for a suitable constant $B$.

To estimate the $\Theta$-term we note that
\begin{equation*}\label{e:log-bound}
\log(z) \leq 1.56 z^{1/8}
\end{equation*}
holds for $z\geq 10^9$ and we therefore have
\begin{equation}\label{e:rm1}
 \sum_{d>z}\frac{\abs{c_d}}{d} \left(\frac{\log(z) + \log(d)}{\zeta(2)} + \abs{A}\right) \leq \frac{1}{\sqrt{z}} \left(\frac{2\cdot 1.56}{\zeta(2)} + 10^{-9/8} \abs A\right) \tilde h\left(\frac 38\right)
\end{equation}
for such $z$.

Next we estimate $\tilde h (\sigma)$. For $\sigma\geq 1/4$ and $t\geq 20$ the inequality
\[
t^{1+\sigma} - t^\sigma -t +1 \geq t^{1+\sigma}/2
\]
holds so that we get
\begin{equation}\label{e:ht-bound}
 \tilde h(\sigma) \leq \left(\prod_{p<10000} \frac{1+(p^\sigma-1)^{-1}(p-1)^{-1}}{(1+p^{-1-\sigma})^2} \right)\frac{\zeta(1+\sigma)^2}{\zeta(2+2\sigma)^2}
\end{equation}
for $\sigma\geq \frac 14$. This gives the bounds $\tilde h(3/8)\leq 19$ and $\tilde h(1/2)\leq 9.4$. For $\abs{A}$ we use the bound
\begin{equation*}
 \abs A \leq \frac{C_0}{\zeta(2)} + 2\abs{\sum_{d\leq 100} \frac{\mu(d)\log(d)}{d^2}} + 2\int_{100}^\infty \frac{\log(t)}{t^2}\,dt \leq 1.8.
\end{equation*}
The error term in \eqref{e:rm2} is thus bounded by $(1.9 + 0.14)\cdot 19/\sqrt{z} \leq 38.8 / \sqrt{z}$ for $z\geq 10^9$. If we use this in \eqref{e:S-bound1}, we arrive at
\begin{equation*}
 \sum_{n\leq z} \frac{\mu(n)^2}{\varphi(n)} = \log(z) +  B + \Theta\left(\frac{58}{\sqrt{z}}\right).
\end{equation*}
 For a proof that $B$ coincides with the constant in \eqref{e:S-asymp} see e. g.
\cite[Lemma 5.4.2]{Bruedern1995}.

\end{proof}

%%%%%%%%%%%%%%%%%%%%%%%%%%%%%%%%%%%%%%%%%%%%%%%%%%%%%%
\begin{proof}[Proof of Theorem \ref{t:wbti2}]
Let $g$ and $Y$ be as in the proof of Theorem \ref{t:wbti1}. First, we consider the case $Y>4\times 10^{18}$. From Proposition \ref{s:wss} we get
\begin{equation*}
\sum_{p\in I} g(p) \leq \frac{Y}{S_1(z)} + \frac{H_1^2(z)}{S_1^2(z)} + \pi(z)
\end{equation*}
for every $z\geq 1$. Here we choose $z=\sqrt{Y}/2$. Then we have $z>10^9$ and Lemma \ref{l:S-asymp} gives 
\begin{multline*}
S_1(z) \geq \log(z) + C_0 + \sum_{p<1000} \frac{\log(p)}{p(p-1)} - 0.002 \\ \geq \log(z) + 1.32  \geq \frac{1}{2} \bigl(\log(Y) + 1.25\bigr).
\end{multline*}
From Lemma \ref{l:H-asymp} we get
\begin{equation*}
H_1(z) \leq \left(\frac {15}{\pi^2} + \frac{47}{\sqrt{10^9}}\right) z \leq 1.53\, z,
\end{equation*}
so we have
\begin{equation*}
\frac{Y}{S_1(z)} \leq 2 \frac{Y}{\log(Y) + 1.25} = 2\frac{Y}{\log(Y)} - 2.5 \frac{Y}{\log(Y)(\log(Y)+1.25)}
\end{equation*}
and
\begin{equation*}
\frac{H_1^2(z)}{S_1^2(z)} \leq 2.4 \frac{Y}{(\log(Y)+1.25)^2}.
\end{equation*}
Using the trivial bound $\pi(z)\leq z$  we thus arrive at
\begin{align*}
\frac{Y}{S_1(z)} + \frac{H_1^2(z)}{S_1^2(z)} + \pi(z) 
  &\leq 2\frac{Y}{\log(Y)} + \frac{\sqrt Y}{2} - \frac{ 0.1Y}{\log(Y)(\log(Y)+1.25)}\\
  &< 2 \frac Y{\log(Y)}.
\end{align*}

Next, we consider the case $1<Y<20000$. From Proposition \ref{s:wes} we obtain the bound
\begin{equation*}
\sum_{p\in I} g(p) \leq Y \prod_{p\leq z}\left(1-\frac{1}{p}\right) + 2^{\pi(z)} + \pi(z).
\end{equation*}
It thus suffices to show that for any such $Y$ there exists a number $z$ such that
\begin{equation*}
 Y \prod_{p\leq z}\left(1-\frac{1}{p}\right) + 2^{\pi(z)} + \pi(z) < 2 \frac{Y}{\log(Y)}
\end{equation*}
holds, but this is carried out in \cite[p. 130]{MV1973}. 

It remains to treat the case $20000\leq Y \leq 4\times 10^{18}$. This can be done numerically: the function $t\mapsto t/\log(t)$ is concave for $t>e^2$. So if for fixed $z$
\begin{equation}\label{e:test}
\frac{Y}{S_1(z)} + \frac{H_1^2(z)}{S_1^2(z)} + \pi(z) < 2\frac Y {\log(Y)},
\end{equation}
holds for $Y\in\{Y_0,Y_1\}$, then it holds for all $Y\in [Y_0,Y_1]$. The inequality \eqref{e:test} was verified with a simple C-routine for all $z\in [50,2\times 10^9)\cap \Z$ and $Y\in\lset 4z^2, 4(z+1)^2\rset$. Therefore, \eqref{e:test} holds for every $Y\in[4\times 50^2, 1.6\times 10^{19}]$ with $z=\sqrt{Y}/2$.
\end{proof}

\section{Improvements}
The results in Theorem \ref{e:wbti2} can be improved by sharpening the estimates for $S_k$ and $H_k$ as pointed out in \cite{Lint1965}. It can also be seen that the equivalent of the stronger result in \cite{MV1973},
\[
\sum_{\substack{p\in I\\p\equiv l\bmod k}} f(p) < \frac{2}{\varphi(k)}\frac{\nrm{f}_{1,I}}{\log(\rho_I(f)/k)},
\]
holds at least for $\rho_I(f) \gg_\eps k^{1+\eps}$ and every $\eps>0$. However, it is not clear whether it can be shown to hold for $\rho_I(f) > k$.

\section*{Acknowledgment}
The author wishes to thank Jens Franke for suggesting this problem. He also wishes to thank the anonymous referee for his/her helpful comments, especially for his/her suggestions to improve the error terms in Lemma \ref{l:H-asymp} and Lemma \ref{l:S-asymp}.

\bibliographystyle{amsalpha}
\bibliography{wbtibib}

\end{document}